\documentclass[10pt]{amsart}
\usepackage{todonotes}
\usepackage{amssymb,mathrsfs}
\usepackage{amsmath}
\usepackage{mathtools}
\usepackage{mathabx}
\usepackage{color}
\usepackage{graphicx}
\graphicspath{{./figures/}}
\usepackage{float}
\usepackage{tikz-cd}
\usepackage{tabularx}
\usepackage{tabulary}
\usepackage{multirow}
\usepackage[T1]{fontenc}
\usepackage{lmodern}
\usepackage{xfrac}

\usetikzlibrary{cd}
\usepackage[bottom]{footmisc}
\usepackage[colorlinks,hyperindex,linkcolor=blue]{hyperref}
\hypersetup{
 pdfauthor = {Lauran E. Toussaint, Thomas O. Rot},
 pdftitle= {Non-linear proper Fredholm maps and the stable homotopy groups of spheres},
 pdfsubject = {}}
\usepackage{enumerate}

\newcommand{\prop}{{\operatorname{prop}}}
\newcommand{\stab}[1]{\mathcal{S}#1}
\DeclareMathOperator{\ind}{ind}

\newcommand{\point}{*}
\newcommand{\K}{{\operatorname{K}}}     

\DeclareMathOperator{\Vect}{Vect}
\DeclareMathOperator{\Det}{Det}
\newcommand{\R}{{\mathbb{R}}}
\newcommand{\Z}{{\mathbb{Z}}}

\newcommand{\N}{\mathbb{N}}

\renewcommand{\H}{\mathbb{H}}
\renewcommand{\S}{\mathbb{S}}

\newcommand{\F}{\mathcal{F}}
\newcommand{\FO}{\mathcal{F}\mathcal{O}}

\newcommand{\id}{{\operatorname{id}}}

\newcommand{\coker}{\operatorname{coker}}
\newcommand{\fr}{\mathrm{fr}}

\renewcommand{\d}{{\operatorname{d}}}
\newcommand{\pr}{{\operatorname{pr}}}
\newcommand{\wtd}{\widetilde}
\newcommand{\tois}{\xrightarrow{\sim}}

\newcommand{\Hom}{{\operatorname{Hom}}}
\newcommand{\GL}{{\operatorname{GL}}}

\newcommand{\colim}{{\operatorname{colim}}}

\newtheorem{lemma}{Lemma}

\newtheorem{theorem}[lemma]{Theorem}
\newtheorem*{theorem*}{Theorem}
\newtheorem{corollary}[lemma]{Corollary}
\newtheorem{definition}[lemma]{Definition}

\begin{document} 
\title{Non-linear proper Fredholm maps and the stable homotopy groups of spheres}
\author{Thomas O. Rot \and Lauran Toussaint}

\begin{abstract}
We classify non-linear proper Fredholm maps between Hilbert spaces, up to proper homotopy, in terms of the stable homotopy groups of spheres. We show that there is a surjective map from the stable homotopy groups of spheres to the set of non-linear proper Fredholm maps up to proper homotopy and determine the non-trivial kernel. We also discuss the case of oriented non-linear proper Fredholm maps which are in bijection with the stable homotopy groups of spheres. 
\end{abstract}

\maketitle

\section{Introduction}

Consider a real, infinite dimensional, separable Hilbert space $\H$. In this paper we give a full classification of (non-linear) proper Fredholm maps $f:\H \to \H$, up to proper Fredholm homotopy, in terms of the stable homotopy groups of spheres. We also give a full classification when the maps, as well as the homotopies are oriented.

Classification results for restricted classes of proper homotopy classes non-linear Fredholm maps have been achieved by several authors in the past, cf.~\cite{BauerFuruta, Berger, Geba, Svarc}. Common in these results is that the differentials of the non-linear Fredholm maps are assumed to land in a contractible subspace of the space of linear Fredholm maps. In \cite{AbbondandoloRot, AbbondandoloRot2}, building on the work~\cite{ElworthyTromba}, complete classification results were obtained for proper homotopy classes of non-linear Fredholm maps of index $n\leq 1$ in terms of computable numerical invariants. 

In this paper we relate non-linear proper Fredholm maps with the stable homotopy groups of spheres. The key idea is that, given any Euclidean space $\R^n$, the direct sum $\R^n \oplus \H$ is isomorphic to $\H$. Consequently we obtain a suspension map:
\begin{equation}\label{eq:Suspension}
\mathcal{S}:[\R^{n+k},\R^k]^\prop \to \F_n^\prop[\H],\quad [f] \mapsto [f\oplus \id_\H],
\end{equation}
where $\F_n^\prop[\H]$ denotes proper homotopy classes of index-$n$ proper Fredholm maps from $\H$ to itself.
It is shown in \cite{Rot} that for $k$ sufficiently large $[\R^{n+k},\R^k]^\prop$ is independent of $k$, and isomorphic to the $n$-th stable homotopy group of spheres $\pi_n^S$. We obtain another suspension map, which is the central object of this paper
\begin{equation}\label{eq:SuspensionMap}
\mathcal{S}:\pi^S \to \F^\prop[\H].
\end{equation}
Here we denote $\pi^S:=\oplus_n\pi_n^S$ and $\F^\prop[\H]=\bigcup_n \F_n^\prop[\H]$.
The main result about this map reads as follows:
\begin{theorem}
  \label{thm:main}
The suspension map induces a bijection
\[\mathcal{S}:\pi^S/{\sim}\tois \F^{\prop}[\H],\]
where $[f]\sim[g]$ if and only if $[f]=[g]$ or $[f]=-[g]$ in $\pi^S$.
\end{theorem}
The group structure on $\pi^S$ does not descend to the quotient. Nevertheless the multiplicative structure, coming from composition of maps, remains. This turns the set $\pi^S/{\sim}$ into a monoid, and the suspension map is an isomorphism of monoids.
Nishida's theorem \cite{Nishida} states that any element of $\pi^S_n$, $n > 0$, is nilpotent implying the following corollary:

\begin{corollary}
Let $f:\H\rightarrow \H$ be a proper Fredholm map of index $n\not= 0$. Then there is a number $k\in \N$ such that $f^k=f\circ\ldots \circ f$ is proper Fredholm homotopic to a non-surjective map. 
\end{corollary}

The inverse of $[f] \in \pi^S$ is represented by $\tau \circ f$, where $\tau$ is a reflection in the first coordinate. 
The suspension of $\tau$ is an invertible linear map which, since $\GL(\H)$ is contractible cf.~\cite{Kuiper}, is proper Fredholm homotopic to the identity. Consequently $\mathcal{S}[f] = \mathcal{S}[\tau \circ f] \in \F^\prop[\H]$, explaining the appearance of the equivalence relation in Theorem~\ref{thm:main}.

We can distinguish between these maps by considering oriented proper Fredholm maps and oriented homotopies between them. We denote the resulting space of homotopy classes by $\FO^\prop[\H]$. In this case Theorem \ref{thm:main} has the following consequence:
\begin{theorem}\label{thm:MainTheoremOriented}
The suspension map induces a bijection
\[ \mathcal{S}:\pi^S \tois \FO^\prop[\H],\]
for any $n \in \Z$.
\end{theorem}
In particular this defines a group structure on $\FO^\prop[\H]$. Both the proofs Theorem \ref{thm:main} and Theorem \ref{thm:MainTheoremOriented} are based on the infinite dimensional version of the Thom-Pontryagin construction obtained in \cite{AbbondandoloRot}. It states that proper Fredholm maps up to proper homotopy correspond to framed submanifolds of $\H$. 

In this setting a framing of an $n$-dimensional submanifold $X \subset \H$ consists of a map
\[ A: \H \to \Phi_n(\H),\]
to the space of Fredholm operators of index $n$, satisfying $\ker A|_X = TX$. The proof of our theorems then boils down to understanding when such a map is homotopic to one of ``finite type'', which means that its image is in $\Phi_n(\infty)\subset \Phi_n(\H)$, the space of linear Fredholm operators that are finite dimensional perturbations of the shift operator.   

To answer this question we introduce a relative version of the parametric Fredholm index, and derive the following generalization of the Atiyah-J\"anich theorem \cite{Atiyah,Janich}:

\begin{theorem}
\label{thm:AtiyahJanichRelative}  
  Let $(X,Y)$ be a pair with $X$ compact and connected and $Y$ closed in $X$, then for any $n$ the Fredholm index
\[ \ind:[(X,Y),(\Phi_n(\H),\Phi_n(\infty))] \to K(X,Y),\]
is a group isomorphism. Moreover, $\ind A = 0$ iff for any $k \in \N$ the map $A$ is homotopic, relative to $Y$ and preserving $\ker A \cap \underline{\R}^{n+k}$, to a map $B:X \to \Phi_n(\infty)$.
\end{theorem}

The organization of the paper is as follows. In Section~\ref{sec:AtiyahJanich} we set up our notation and prove Theorem~\ref{thm:AtiyahJanichRelative}. In Section~\ref{sec:NonLin} we discuss the relationship between non-linear Fredholm maps and framed bordism. We prove Theorem~\ref{thm:main} in Section~\ref{sec:MainCase}. In Section~\ref{sec:OrientedCase} we discuss how to assign orientations to Fredholm maps and prove Theorem~\ref{thm:MainTheoremOriented}.

\subsection*{Acknowledgements} 
We are grateful to Alberto Abbondandolo and Michael Jung for many fruitful discussions.
Lauran Toussaint is funded by the Dutch Research Council (NWO) on the project ``proper Fredholm homotopy theory'' (OCENW.M20.195) of the research programme Open Competition ENW M20-3.

\section{The relative Atiyah-J\"anich Theorem}
\label{sec:AtiyahJanich}
We fix once and for all an isomorphism 
\[\H \cong \ell^2 := \{(x_i)_{i\in \N} \mid x_i\in\R, \sum_{i=1}^\infty x_i^2 < \infty\}.\]
Of course, since $\H$ is a separable infinite dimensional Hilbert space such an isomorphism always exists, and is equivalent to the choice of a basis $\{e_i\}_{i \in \N}$. This yields a canonical inclusion of subspaces
\begin{equation}\label{eq:Subspaces}
 \R^n \subset \R^\infty \subset \H,
\end{equation}
where $\R^\infty := \colim_n\, \R^n$. 

The space of Fredholm operators is denoted by $\Phi_n(\H)$. The shift operator of index $n$ is defined by 
\[ S_n:\H \to \H,\quad S_n(e_i) := e_{i-n},\]
where we set $e_k=0$ for $k \leq 0$. Analogous to Equation~\eqref{eq:Subspaces} we have subspaces of linear maps:
\[ \Hom(\R^{n+k},\R^k) \subset \Phi_n(\infty) \subset \Phi_n(\H),\]
where $\Phi_n(\H)$ denotes the space of Fredholm operators of index $n$, and 
$\Phi_n(\infty) := \colim_k\, \Hom(\R^{n+k},\R^k)$ under the obvious inclusions.
We say that a Fredholm operator is of finite type if it is contained in the union
\begin{equation*}
\Phi(\infty) := \bigcup_{n \in \Z} \Phi_n(\infty).
\end{equation*}
More explicitly this means that $A \in \Phi(\H)$ is of finite type if there exists a decomposition \[A:\H = \R^{n+k} \oplus (\R^{n+k})^\perp \to \H = \R^k \oplus (\R^k)^\perp\] in which we can write
\begin{equation}\label{eq:FiniteType}
A = \begin{pmatrix} \widehat{A} & 0 \\ 0 & S_n \end{pmatrix}\quad \text{with } \widehat{A} \in \Hom(\R^{n+k},\R^k).
\end{equation}

Let $V \subset \H$ be a closed subspace with finite codimension and consider
\[ \Phi_V(\H) := \{ A \in \Phi(\H) \mid \ker A \cap V = 0\}.\]
This is an open subset of $\Phi(\H)$, and any $A \in \Phi(\H)$ is an element of $\Phi_{(\ker A)^\perp}(\H)$. The following lemma lists some of its basic properties which are easily verified, see also \cite[Appendix A]{Atiyah}.

\begin{lemma}\label{lem:BasicProperties}
Let $X$ be compact, and $A:X \to \Phi(\H)$ a continuous map. Then we have:
\begin{enumerate}[(i)]
    \item There exists a closed subspace $V \subset \H$ with finite codimension such that $A_x \in \Phi_V(\H)$ for all $x \in X$, and
    \[ \coker_VA := \bigcup_{x\in X} \H/A_x(V),\]
    is a (finite rank) vector bundle over $X$. We may choose $V=(\R^{{n+k}})^\perp$ for ${n+k}$ sufficiently large. 
  \item If $W\subset V$ is of finite codimension, then $\Phi_V(\H) \subset \Phi_W(\H)$ and for all $A\in \Phi_V(\H)$ we have the exact sequence
    \[
      0\to V/W\to\coker_WA\to \coker_V A\to 0
    \]
    and therefore 
    \[ \coker_WA \cong\coker_V A \oplus \underline{\R}^d,\]
    where $\underline{\R}^d$ denotes the trivial bundle of rank $d := \dim V/W$ over $X$.
\end{enumerate}
\end{lemma}

These properties imply that, given $A:X \to \Phi(X)$, the virtual bundle
\[ \ind A := [\H/V] - [\coker_V A] \in \K(X),\]
is well-defined and independent of $V$. Furthermore the K-theory class $\ind A$ depends only on the homotopy class of $A$. The celebrated Atiyah-J\"anich Theorem \cite{Atiyah,Janich} states that the index classifies maps into the Fredholm operators up to homotopy.

\begin{theorem}\label{thm:AtiyahJanich}
If $X$ is compact, the Fredholm index induces a group isomorphism
\[ \ind:[X,\Phi(\H)] \to \K(X).\]
\end{theorem}

We are interested in homotopy classes of maps $A:(X,Y)\rightarrow (\Phi(\H),\Phi(\infty))$, i.e.~where the Fredholm operators are of finite type over $Y$. In order to derive an Atiyah-J\"anich theorem in this setting we recall the following quotient construction for vector bundles.
Given a pair $(X,Y)$, where $X$ is compact and  $Y$ is closed in $X$ consider
\[\Vect(X,Y) := \left\{ (E,{\varphi}) \mid E \to X \text{ a vector bundle}, {\varphi}:E\bigr|_Y \tois \underline{\R}^k\bigr|_Y\text{ a trivialization}\right\}.\]
The trivialization defines a projection
\[ \pi:E|_Y \to \R^k,\quad e \mapsto (\pr_2 \circ \varphi)(e),\]
where $\pr_2:\underline{\R}^k=Y \times \R^k \to \R^k$ denotes the second coordinate projection. As such we obtain a map:
\begin{equation}\label{eq:VectorbundleQuotient}
\Vect(X,Y) \to \Vect(X/Y,\{y\}),\quad (E,\varphi) \mapsto E/\varphi := E/{\sim},
\end{equation}
where $e \sim e'$ if and only if $e,e'\in E|_Y$ and $\pi(e) = \pi(e')$.

Going back to the Atiyah-J\"anich theorem; consider a map $A:(X,Y) \to (\Phi_n(\H),\Phi_n(\infty))$, and $V$ as in Lemma~\ref{lem:BasicProperties}.
It follows from Equation \eqref{eq:FiniteType} that, for any $y \in Y$, the composition of the inclusion and quotient map
\[ \R^k \hookrightarrow \H \xrightarrow{\pi} \H/A_y(V),\]
is an isomorphism. Therefore, we have canonical trivializations:
\[
\varphi:\underline{V}^\perp|_Y \tois \underline{\R}^{n+k},\quad \text{and} \quad \psi:\coker_VA|_Y \tois \underline{\R}^k,
\]
for some $k$ sufficiently large. As such we define the relative Fredholm index by
\begin{equation}\label{eq:IndexRelative}
\ind:[(X,Y),(\Phi(\H),\Phi(\infty))] \to K(X,Y),\quad [A] \mapsto [{\underline{V}}^\perp/{\varphi}] - [\coker_VA/\psi] - [\underline{\R}^n].
\end{equation}
We now prove Theorem~\ref{thm:AtiyahJanichRelative}, which we restate here for convenience. 

\begin{theorem*}
Let $(X,Y)$ be a pair with $X$ compact and connected and $Y$ closed in $X$, then for any $n$ the Fredholm index
\[ \ind:[(X,Y),(\Phi_n(\H),\Phi_n(\infty))] \to K(X,Y),\]
is a group isomorphism. Moreover, $\ind A = 0$ iff for any $k \in \N$ the map $A$ is homotopic, relative to $Y$ and preserving $\ker A \cap \underline{\R}^{{n+k}}$, to a map $B:X \to \Phi_n(\infty)$. 
\end{theorem*}

\begin{proof}
For each $n$ we choose the shift operator $S_n$ as the base-point of $\Phi_n(\H)$. For a compact pointed space $(X,x_0)$ Theorem \ref{thm:AtiyahJanich} implies that
\[ [(X,x_0),(\Phi_n(\H),S_n)] \cong \wtd{K}(X),\]
where the right-hand side denotes the reduced $K$-theory of $X$.

Since $Y$ is compact its image is contained in $\Hom(\R^{n+k},\R^k) \subset \Phi_n(\infty)$ for some $k$ sufficiently large. The set $\Hom(\R^{n+k},\R^k)$ is contractible and its inclusion into $\Phi_n(\H)$ is a cofibration. Hence there exists a homotopy $h_t:\Phi_n(\H) \to \Phi_n(\H)$ satisfying $h_0 = \id$ and $h_1(B)=S_n$ for all $B\in \Hom(\R^{n+k},\R^k)$. 
This implies
\[[(X,Y),(\Phi_n(\H),\Phi_n(\infty))] \cong [(X,Y),(\Phi_n(\H),S_n)].\]
Putting these observations together we obtain:
\begin{align*}
[(X,Y),(\Phi_n(\H),\Phi_n(\infty))] &\cong [(X,Y),(\Phi_n(\H),S_n)]\\
&\cong [(X/Y,[Y]),(\Phi_n(\H),S_n)] \cong \wtd{K}(X/Y) = K(X,Y).
\end{align*}
Note that the resulting isomorphism is precisely the index as defined in Equation \eqref{eq:IndexRelative}.
For the second claim observe that if $A:X \to \Phi(\infty)$ then $\ind A = 0 \in K(X,Y)$. Since the index is preserved under homotopy the first implication follows.

For the converse assume that $A:X \to \Phi_n(\H)$ satisfies $\ind A= 0 \in K(X,Y)$. Then {there exists $k\in \N$ such that $V=(\R^{n+k})^\perp$ is as in Lemma~\ref{lem:BasicProperties}}, and $\coker_VA$ is trivial.
In fact, together with Equation~{\eqref{eq:FiniteType}}, this shows that the bundle 
\[ {(A(V))^\perp := \bigcup_{x \in X} (A_x(V))^\perp}, \]
is trivial and ${(A(V))^\perp}|_Y = \underline{\R}^k$. {The bundle $A(V)$ is also trivial, as it is an infinite dimensional Hilbert bundle}. Hence, there exists $\tau:(X,Y) \to (\GL(\H),\id_\H)$ such that ${\tau (A(V))^\perp} = \underline{\R}^k$ {and $\tau A(V)=({\underline \R}^k)^\perp$} on the whole of $X$.

Kuiper's theorem~\cite{Kuiper} implies that $\GL(\H)$ deformation retracts onto $\id_\H$. As such, $A$ and $\tau A$ are homotopic relative to $Y$ and preserving the kernel. With respect to the decompositions $\H = \R^{n+k} \oplus (\R^{n+k})^\perp$ and $\H = \R^k \oplus (\R^k)^\perp$ we can write
\[ \tau A = \begin{pmatrix} \widehat{\tau A} & 0 \\ C & D \end{pmatrix},\]
where $C:(X,Y) \to (\Hom(\R^{n+k},\R^k),0)$ and $D:(X,Y) \to (\GL( (\R^{n+k})^\perp, (\R^k)^\perp), S_n)$.

Again using Kuiper's theorem and the identification $(GL((\R^{n+k})^\perp,(\R^k)^\perp),S_n) \cong (\GL(\H),\id_\H)$ we see there exists a homotopy $D_t$ from $D$ to the constant map with value $S_n$.
In turn this gives a homotopy
\[
\begin{pmatrix}
\widehat{\tau A} & 0 \\
(1-t)C & D_t
\end{pmatrix},\]
relative to $Y$, from $\tau A$ to 
\[
B := \begin{pmatrix}
\widehat{\tau A} & 0 \\ 
0 & S_n
\end{pmatrix}.
\]
Note that $\ker \tau A \cap \underline{\R}^{n+k} = \ker \widehat{\tau A} \cap \ker C$, so that the homotopy also preserves $\ker A \cap \underline{\R}^{n+k}$.
\end{proof}

When $Y = \emptyset$ we recover Theorem \ref{thm:AtiyahJanich}. Furthermore, if we do not assume $X$ to be connected nor $n$ to be fixed the above proposition implies that
\[ [(X,Y),(\Phi(\H),\Phi(\infty))] \simeq K(X,Y) \times \operatorname{Im}\left(\iota^*:H^0(X;\Z) \to H^0(Y;\Z)\right).\]
Here $\iota^*:H^0(X;\Z) \to H^0(Y;\Z)$ is induced by the inclusion $\iota:Y \to X$.

\section{Non-linear Fredholm maps and framed bordism}
\label{sec:NonLin}
\subsection{Non-linear maps} 
The space of smooth (non-linear) Fredholm maps from Hilbert space to itself is denoted by $\F(\H)$. There are inclusions
\[ C^\infty(\R^{n+k},\R^k) \subset \F_n(\infty) \subset \F_n(\H),\]
where the second inclusion is given by sending $f$ to $f\oplus \id_\H$ and $\F_n(\infty) := \colim_k C^\infty(\R^{n+k},\R^k)$. A Fredholm map is said to be of finite type if it is contained in the union
\[\F(\infty) := \bigcup_{n \in \Z} \F_n(\infty).\]
{Recall that a map is proper if preimages of compact sets are compact.} Two proper Fredholm maps $f,g\in \F^\prop(\H)$ are said to be proper homotopic if there exists a homotopy $h$ which is proper as a map $h:I\times \H \to \H$.
We denote by $\F^\prop[\H]$ the resulting set of homotopy classes. Note that proper homotopy is strictly stronger than the standard notion of homotopy, or even homotopy through proper maps. As shown in \cite[Corollary 1.2]{Rot} there is a bijection
\[ [\S^{n+k-1},\S^{k-1}] \tois [\R^{n+k},\R^k]^\prop.\]

 Under this bijection taking the suspension of a sphere (and maps between them) corresponds to taking the product of Euclidean space with $\R$ (and $\oplus \id_\R$ for maps). 
This yields the following commutative diagram:
\begin{equation*}
\begin{tikzcd}[
ar symbol/.style = {draw=none,"\textstyle#1" description,sloped},
isomorphic/.style = {ar symbol={\cong}},
]
\left[\S^{n+k-1},\S^{k-1}\right] \arrow[d,"\Sigma"] \arrow[r,"\sim"]& \left[\R^{n+k},\R^k\right]^\prop \arrow[d,"\oplus \id_\R"] \arrow[r,"\mathcal{S}"] & \F_n^\prop[\H]\\
\left[\S^{n+k},\S^k\right] \arrow[r,"\sim"] & \left[\R^{n+k+1},\R^{k+1}\right]^\prop \arrow[ur,swap,"\mathcal{S}"]&\\
\end{tikzcd}
\end{equation*}
Hence by the Freudenthal suspension theorem we obtain a well-defined map:
\begin{equation}\label{eq:SuspensionMaps}
\mathcal{S}:\pi_n^S \to \F^\prop_n[\H],
\end{equation}
which is precisely the map in Equation~\eqref{eq:SuspensionMap}.

\subsection{Framed submanifolds and cobordisms}
A framing of a manifold $X^n \subset \R^{n+k}$ is a trivialization of its normal bundle, or equivalently a map
\begin{equation}\label{eq:ClassicalFraming}
A:X \to \Hom(\R^{n+k},\R^k)\quad \text{satisfying}\quad \ker A = TX.
\end{equation}

Since the space of linear maps is contractible $A$ always extends to a map $A:\R^{n+k}\to\Hom(\R^{n+k},\R^k)$, defined on the whole of $\R^{n+k}$.
Often we are interested in such manifolds up to stabilization; a stably framed manifold consists of a closed $n$-dimensional submanifold $X \subset \R^\infty$, together with a map
\begin{equation}
  \label{eq:ExtendedFraming}
  A:\R^\infty \to \Phi_n(\infty) \quad \text{satisfying} \quad \ker A = TX,
\end{equation}
and depending only on the first $n+k$ coordinates of $\R^\infty$ for some $k\in \N$. Stably framed cobordisms are defined analogously and the resulting cobordism ring $\Omega^\fr(\R^\infty)$ coincides with the usual one in the literature, e.g. \cite[Definition 8.12]{DavisKirk}.
Submanifolds of Hilbert space are framed using $\Phi(\H)$ which has the homotopy type of $\Z \times BO$. As such, an extension does not always exist, and instead it becomes part of the definition:

\begin{definition}
  \label{def:framing}
A framing of an $n$-dimensional closed submanifold $X \subset \H$ is a continuous map
\[ A:\H \to \Phi_n(\H),\]
such that $\ker A_x = T_xX$ for every $x \in X$. 
\end{definition}
We use the convention that the empty set is a manifold of any dimension. Hence the definition above also makes sense for negative $n$, in which case a framing is just a map $A:\H \to \Phi_n(\H)$. Framed cobordisms are defined analogously. 
\begin{definition} 
\label{def:cobordism}
  A cobordism between framed $n$-dimensional submanifolds $(X_0,A_0),(X_1,A_1) \subset \H$ is a compact $(n+1)$-dimensional manifold $W \subset I \times \H$ with boundary such that: 
\[
\partial W \subset \partial I \times \H, \qquad W \cap \left([0,\varepsilon) \times \H\right) = [0,\varepsilon) \times X_0, \qquad W \cap \left((1-\varepsilon,1]\times \H\right) = (1-\varepsilon,1] \times X_1,
\]
for some $\varepsilon > 0$, together with a framing; meaning a continuous map
\[
B: I\times \H \rightarrow \Phi_{n+1}(\R\oplus \H, \H)
\]
satisfying
\[
\ker\,B_{(t,x)} = T_{(t,x)} W \qquad \forall (t,x)\in W,
\]
and for every $t_0\in [0,\epsilon)$, $t_1\in (1-\epsilon,1]$ and $x\in \H$ we have 
\[
B_{(t_0,x)}(s,u) = A_0 (u) \qquad  B_{(t_1,x)}(s,u) = A_1(u).
\]
The set of framed submanifolds up to framed cobordism is denoted by $\Omega^\fr(\H)$. 
\end{definition}
Any submanifold of $\R^\infty$ can be seen as a submanifold of $\H$ and any framing as in~\eqref{eq:ExtendedFraming} extends uniquely to a framing as in Definition~\ref{def:framing}. Hence we obtain a suspension map for framed submanifolds:
\begin{equation}\label{eq:SuspensionSubmanifold}
\stab:\Omega^\fr(\R^\infty) \to \Omega^\fr(\H).
\end{equation}
We say that a framed submanifold $(X,A) \subset \H$ is of finite type if it is in the image of the map above. 

The classical stable Pontryagin-Thom construction gives an isomorphism
\[\pi^S\cong \Omega^\fr(\R^\infty) ,\]
between the stably framed cobordism ring and the stable homotopy groups of spheres, see for example \cite[Corollary 8.10]{DavisKirk}.
Similarly, to a proper Fredholm map $f\in \F^\prop(\H)$ we assign the submanifold $f^{-1}(y) \subset \H$, for some regular value $y$, which is canonically framed by $\d f: \H \to \Phi(\H)$.
As proven in~\cite[Theorem 2]{AbbondandoloRot} this induces an isomorphism
\[\F^\prop[\H] \cong \Omega^\fr(\H).\]
Under these isomorphisms suspending (homotopy classes) of maps, as in Equation~\eqref{eq:SuspensionMaps}, corresponds to suspending framed submanifolds as in Equation~\eqref{eq:SuspensionSubmanifold}. The situation is summarized in the following diagram:
\begin{equation}\label{eq:Diagram}
\begin{tikzcd}[
ar symbol/.style = {draw=none,"\textstyle#1" description,sloped},
isomorphic/.style = {ar symbol={\cong}},
]
\Omega^\fr(\R^\infty) \arrow[d,isomorphic] \arrow[r,"\mathcal{S}"] & \Omega^\fr(\H) \arrow[d,isomorphic]\\
\pi^S \arrow[r,"\mathcal{S}"] & \F^\prop[\H] \\
\end{tikzcd}
\end{equation}

\section{Proof of Theorem \ref{thm:main}}
\label{sec:MainCase}
In light of the diagram above, Theorem \ref{thm:main} is an immediate consequence of the following result.
\begin{theorem}{\label{thm:MainTheoremCobordism}}
The $\H$-suspension map induces a bijection
\[\stab:\Omega^\fr(\R^\infty)/{\sim} \to \Omega^{\fr}(\H),\]
where $[X_0,A_0] \sim [X_1,A_1]$ if and only if $[X_0,A_0] = \pm [X_1,A_1]$. 
\end{theorem}
\begin{proof}
\underline{Simplification:} 
We start by showing that, without loss of generality, we may assume all of of submanifolds to be contained in a finite dimensional Euclidean space. Note that this does not mean they are of finite type. To be precise we have:
\begin{itemize}
\item Any {$n$-dimensional} framed submanifold $(X_0,A_0)$ in $\H$ is framed cobordant to a framed submanifold $(X_1,A_1)$ satisfying $X_1 \subset {\R^{n+k}}$ for some {$k\in \N$}.
\item Let $(X_0,A_0),(X_1,A_1)$ be framed cobordant framed submanifolds with $X_0,X_1\subset \R^{n+k}$ for some ${k}\in \N$. Then there exists a framed cobordism $(W,B)$ from $(X_0,A_0),(X_1,A_1)$ satisfying $W\subset I\times \R^{n+k}$ for some $k\in\N$. 
\end{itemize}

To prove the first claim let $i:X\rightarrow \H$ be the inclusion and choose an embedding $j:X\rightarrow \H$ such that $j(X)\subset \R^{n+k}$ for some $k\in \N$ using Whitney's embedding theorem. Let $h_1,h_2:I\rightarrow X$ be given by
\[
h_1(t,x)=(1-t)i(x)+tS_{-N}i(x)\qquad h_2(t,x)=(1-t)S_{-N}(i(x))+tj(x).
\]
Both $i_1$ and $i_2$ are isotopies of embeddings. Therefore there exists a diffeomorphism $\varphi:\H\rightarrow \H$, isotopic to the identity, such that $\varphi(X)\subset \R^{n+k}$. By \cite[Lemma 3.7]{AbbondandoloRot2} the framed submanifold $(X,A)$ is framed cobordant to ($\varphi(X),A\circ \varphi d\varphi^{-1})$. The second claim is proved similarly, and we omit the details.

\underline{Surjectivity:}
Let $(X,A)$ be an $n$-dimensional framed submanifold of $\H$ such that $X\subset \R^{n+k}$ for some ${k} \in \N$, and $Z \subset \R^{{n+k}}$ a closed ball containing $X$. Since $Z$ is compact and contractible
\[ \K(Z) \cong \K(\point) \cong \Z,\]
corresponding to the Fredholm index of shift operators. 

By Theorem~\ref{thm:AtiyahJanichRelative} this implies that $A|_Z$ is homotopic to a map whose image is contained in $\Phi_n(\infty)$. Moreover, since $TX \subset \underline{\R}^{n+k}$ we can choose the homotopy to preserve $\ker A|_X$. The homotopy extends to the whole of $\H$, using that $Z$ is a deformation retract of $\H$. Thus, by~\cite [Remark 3.5]{AbbondandoloRot2}, $(X,A)$ is framed cobordant to a framed submanifold of finite type.

\underline{Injectivity:} Let $(W,B)$ be a framed cobordism between finite type framed submanifolds $(X_0,A_0)$ and $(X_1,A_1)$, and suppose that no such cobordism of finite type exists. We will show this implies that $[(X_0,A_0)] = - [(X_1,A_1)] \in \Omega^\fr(\R^\infty)$.
Assume that $W \subset I \times \R^{{n+k}}$ for some ${k} \in \N$, and fix a closed ball $Z \subset \R^{{n+k}}$ such that $W \subset I \times Z$.
Since $Z$ compact and contractible
\[ K(I \times Z, \partial I \times Z ) \cong K(I,\partial I ) \cong \Z/2\Z.\]
By Theorem \ref{thm:AtiyahJanichRelative} the assumption that $(W,B)$ cannot be of finite type implies that $\ind B|_Z$ corresponds to the non-trivial element in $\Z/2\Z$.

Fix a path $\tau:I\to \GL(\H)$ be a path satisfying $\tau(0) = \id_\H$ and $\tau(1) = (-1) \oplus \id_\H$.
Then, consider the framed cobordism
\[(I \times X_1,\tau A_1),\]
where $\tau A_1(t,x):= \tau(t) A_1(x)$. 
The index of $\tau A_1|_{I \times Z}$ corresponds to the non-trivial element of $\Z/2\Z$. To see this, recall Equation~\eqref{eq:IndexRelative} and note that in our case the bundle $\coker_V \tau A_1/\psi$ (for $V$ as in Lemma \ref{lem:BasicProperties} such that $V^\perp = \R^{n+k}$) is non-trivial since it has monodromy
\[ \tau(1)|_{\R^{n+k}} = (-1) \oplus \id_{\R^{n+k -1}}.\]
Also note that $(X_1,\tau(1)A_1)$ is again of finite type, and that since $\tau(1)|_{\R^{n+k}}$ is orientation reversing we have $[(X_1,\tau(1) A_1)] = - [(X_1,A_1)] \in \Omega^\fr(\R^\infty)$.

When composing framed cobordisms, the framings get concatenated (in the interval direction). 
As such, the framing $B \circ \tau A_1$ of the composition $(W,B) \circ (I \times X_1,\tau A_1)$ satisfies:
\[ \ind (B \circ \tau A_1) = \ind B + \ind \tau A_1 = 1 + 1 = 0,\]
using again the identification $K(I,\partial I) \cong \Z/2\Z$.
Together with Theorem \ref{thm:AtiyahJanichRelative} and the observations above this implies
\[
  [(X_0,A_1)] = [(X_1,\tau(1)A_1)] = -[(X_1,A_1)] \in \Omega^\fr(\R^\infty).
\]
\end{proof}

\section{Oriented Fredholm maps and the proof of Theorem~\ref{thm:MainTheoremOriented}}
\label{sec:OrientedCase}

To define oriented Fredholm maps we consider the determinant line bundle  $\Det \to \Phi(\H)$,
whose fiber over $A \in \Phi(\H)$ equals
\[\Det_A = \det(\ker A) \otimes \det(\coker A)^*,\]
where $\det V := \Lambda^{\operatorname{top}}V$, see also \cite{Quillen}. Given a map $A:\H \to \Phi(\H)$ the pullback bundle $A^*\Det$ is trivializable since $\H$ is contractible. Thus we can define an orientation of $A$ to be a choice of trivialization. 
In particular an orientation of a Fredholm map $f$ is an orientation of its differential
\[ \d f: \H \to \Phi(\H).\]

A homotopy between oriented proper Fredholm maps $f$ and $g$ is an oriented proper Fredholm map $h:I \times \H \to \H$ satisfying
\[ h|_{\{0\}\times \H} = f\quad \text{and}\quad h|_{\{1\}\times \H} = g,\]
as oriented maps. The set of all such homotopy classes is denoted by $\FO^\prop[\H]$.

Given a Fredholm map of finite type $f \in \F_{{n}}(\infty)$ there exists {$k\in \N$} and an exact sequence of vector spaces:
\[ 0 \to \ker \,\d f_{x} \xrightarrow{i} {\R}^{n+k} \xrightarrow{\d f_{x}|_{{\R}^{n+k}}} {\R}^k \xrightarrow{\pi} \coker \d f_{x} \to 0,\]
for any $x \in \H$.
As explained in \cite[Equation 5.6]{AbbondandoloMajer} this sequence yields an isomorphism
\begin{align}\label{eq:DeterminantIso}
\Det (\ker \d f) \otimes \Det (\coker \d f)^* &\tois \Det(\underline{\R}^{n+k}) \otimes \Det(\underline{\R}^k)^*\\
\alpha \otimes (\pi_* \gamma)^* &\mapsto (\iota_*\alpha \wedge \beta ) \otimes (\gamma \wedge \d f_*\beta)^*,\nonumber
\end{align}
where $\alpha$ generates $\Det(\ker \d f)$, $\iota_*\alpha \wedge \beta$ generates $\Det(\underline{\R}^{n+k})$, and $\gamma \wedge \d f_*\beta$ generates $\Det(\underline{\R}^k)$. Since $\underline{\R}^{n+k}$ and $\underline{\R}^k$ are canonically oriented this determines an orientation of $f$. 
This implies there exists suspension map:
\[\mathcal{S}:\F^\prop(\infty) \to \FO^\prop(\H).\]
The proof of Theorem \ref{thm:MainTheoremOriented} now follows from Theorem \ref{thm:main} together with the following lemma:
\begin{lemma}
Two maps $f,g \in \F^\prop_{{n}}(\infty)$ are proper homotopic if and only if $\mathcal{S}f$ and $\mathcal{S}g$ are oriented homotopic.
\end{lemma}
\begin{proof}
If $h$ is a homotopy from $f$ to $g$, then its suspension $\mathcal{S}h$ is oriented by Equation~\eqref{eq:DeterminantIso} and hence defines an oriented homotopy from $\mathcal{S}f$ to $\mathcal{S}g$.

Conversely, let $h:I \times \H \to \H$ be an (oriented proper) homotopy from $\mathcal{S}f$ to $\mathcal{S}g$.
By the infinite dimensional Thom-Pontryagin construction \cite[Theorem 2]{AbbondandoloRot} the homotopy corresponds to a framed cobordism $(W,B)$ between the framed submanifolds induced by $f$ and $g$. The fact that $h$ is oriented means that the framing $B:I \times \H \to \Phi(\R\oplus\H,\H)$ is oriented.
As such, to show that $f$ and $g$ are homotopic it suffices to homotope $(W,B)$ relative to its boundary to a finite type cobordism.

The argument is similar to the proof of Theorem~\ref{thm:MainTheoremCobordism}. We may assume without loss of generality that $W \subset I\times\R^{{n+k}} \subset I\times\H$ for some $k \in \N$, and it suffices to homotope $B$ relative to the boundary to a framing of finite type, preserving the kernel along $W$. Let $Z\subset\R^{{n+k}}$ be a closed ball such that $W\subset I\times Z$. 

According to Lemma \ref{lem:BasicProperties} there exists an $N\in \N$ such that $V=(\R^N)^\perp$ satisfies $\ker B\cap V=0$. Without loss of generality we may choose $N=n+k$ by taking $k$ sufficiently large above. The condition $\ker B\cap V=0$ implies that the quotient map $\pi:\H \to \H/V$ restricts to an injection $\ker B \to \H/V$.
This gives an exact sequence of vector spaces
\[0 \to \ker B_{{x}} \to \H/V \to \coker_V B_{{x}} \to \coker B_{{x}} \to 0,\]
for any $x \in \H$, and where the map $\H/V \to \coker_V B$ is induced by $B$.
As in Equation~\eqref{eq:DeterminantIso} this induces an isomorphism 
\begin{align*}
B^*\Det &\cong \Det(\H/V)\otimes \Det(\coker_V B)^*.
\end{align*}
Now the canonical map $V^\perp=\R^{n+k}\rightarrow \H/V$ is an isomorphism, hence $\H/V$ is canonically oriented.
Thus the bundle $\coker_V B$ admits an orientation which is compatible with the 
orientation obtained from the trivialization at the boundary $I\times Z$
\[\psi:\coker_V B|_{\partial I \times Z} \tois \underline{\R}^{{n+k}}.\]
Therefore, the vector bundle $\coker_V B/\psi \to \S^1 \times Z$,
obtained from the quotient construction in Equation \eqref{eq:VectorbundleQuotient} is orientable and hence trivial. Then, Equation \eqref{eq:IndexRelative} and Theorem \ref{thm:AtiyahJanichRelative} imply that $B$ is homotopic to a finite type framing over $Z$. As $I\times \H$ deformation retracts to $\partial I\times \H\cup I\times Z$, we obtain a finite type framed bordism. Therefore $f$ and $g$ are proper homotopic. 
\end{proof}

\bibliographystyle{abbrv}
\bibliography{Bibliography}
\end{document}